\documentclass[11pt]{article}

\usepackage{times}
\usepackage{amsthm}
\usepackage{amsmath}
\usepackage{amssymb}
\usepackage{mathrsfs}
\usepackage{amsfonts}
\def\A{{\mathcal A}}

\def\Th{\Theta}

\def\N{\mathbb{N}}

\def\X{\mathscr X}

\def\C{\mathcal{C}}

\def\B{{\mathcal B}}

\def\H{\mathcal H}
\def\K{\mathcal K}

\def\P{\mathcal P}

\def\S{\mathcal S}

\def\Ai{\mathcal A^\infty}

\def\I{{\rm 1\kern-.26em I}}

\def\sp{\mathop{\mathrm{sp}}\nolimits}
\def\spe{{\rm sp}_{{\rm ess}}}

\def\1{\mathfrak{1}}
\def\0{\mathfrak{0}}

\def\p{\parallel}

\def\<{\langle}
\def\>{\rangle}

\def\Op{\mathrm{Op}}
\def\({\left(}
\def\){\right)}
\def\[{\left[}
\def\]{\right]}

 \newtheorem{thm}{Theorem}[section]
 \newtheorem{cor}[thm]{Corollary}
 \newtheorem{lem}[thm]{Lemma}
 \newtheorem{prop}[thm]{Proposition}
 \newtheorem{rem}[thm]{Remark}
\begin{document}

\title{On the Essential Spectrum of Phase-Space\\ Anisotropic Pseudodifferential Operators}

\author{Marius M\u antoiu }

%\date{\today}
\maketitle

\footnote*{


%\begin{quote}
\begin{itemize}
\item[] Departamento de Matem\'aticas, Universidad de Chile, Las Palmeras 3425, Casilla 653,
Santiago, Chile, \\
Email: {\tt mantoiu@uchile.cl}

\end{itemize}}

\bigskip

%\textbf{2010 Mathematics Subject Classification: Primary 35S05, 81Q10, Secundary 46L55, 47C15.}
%\newline
%\textbf{Key Words:} Differential operator, spectrum, Rieffel quantization, $C^*$-algebra,
%noncommutative dynamical system.}

 Supported by {\it N\'ucleo Cientifico ICM P07-027-F
"Mathematical Theory of Quantum and Classical Magnetic Systems"}.
\maketitle

\begin{abstract}
A phase-space anisotropic operator in $\H=L^2(\mathbb R^n)$ is a
self-adjoint operator whose resolvent family belongs to a natural
$C^*$-completion of the space of H\"ormander symbols of order
zero. Equivalently, each member of the resolvent family is
norm-continuous under conjugation with the Schr\"odinger unitary
representation of the Heisenberg group. The essential spectrum of
such a phase-space anisotropic operator is the closure of the
union of usual spectra of all its "phase-space asymptotic
localizations", obtained as limits over diverging ultrafilters of
$\mathbb R^{n}\times\mathbb R^n$-translations of the operator. The
result extends previous analysis of the purely configurational
anisotropic operators, for which only the behavior at infinity in
$\mathbb R^n$ was allowed to be non-trivial.
\end{abstract}

%-------------------------------------------------------------------------------------------------------
\section{Introduction and main results}\label{duci}
%-------------------------------------------------------------------------------------------------------

We are going to study self-adjoint operators acting in the complex Hilbert space $\H:=L^2(\X)$,
where $\X$ is an $n$-dimensional real vector space.
Let us also set $\Xi:=\X\times\X^*$, where $\X^*$ denotes the dual of $\X$. For reasons coming from physics, we are going to
call the spaces $\X$,$\,\X^*$ and $\Xi$ {\it the configuration}, {\it the momentum} and {\it the phase space}, respectively.
On $\Xi$ there is a canonical symplectic form given by $[\![X,Y]\!]=[\![(x,\xi),(y,\eta)]\!]:=y\cdot\xi-x\cdot\eta$,
in terms of the duality $\X\times\X^*\ni(z,\zeta)\mapsto z\cdot\zeta:=\zeta(z)\in\mathbb R$.

Our main result will be a formula giving the essential spectrum $\spe(H)$ of operators $H$ affiliated
to a remarkable $C^*$-algebra $\mathbf B^0(\H)$ of bounded linear operators in $\H$.
Affiliation means that the resolvent family $\{(H-z)^{-1}\mid z\in\mathbb C\setminus\mathbb R\}$
of $H$ belongs to $\mathbf B^0(\H)$. By a straightforward
application of the Stone-Weierstrass Theorem this implies actually that $\varphi(H)$ (constructed by the usual functional
calculus) belongs to $\mathbf B^0(\H)$ for each continuous function $\varphi:\mathbb R\rightarrow\mathbb C$
which vanishes at infinity. We send to \cite{ABG} or to \cite[Sect.2.1]{GI1} for more on this concept,
which is different from the one introduced by Woronowicz \cite{Wo}.

The above mentioned formula will involve a certain type of limits of the operator $H$ along
suitable filters of the phase space $\Xi$.

To define $\mathbf B^0(\H)$, we introduce first some notations. We set $\mathbf B(\H)$ for the $C^*$-algebra of
linear bounded operators in $\H$ and $\mathbf C_0(\H)$ for its ideal of compact operators.
There is a unitary projective representation $W:\Xi\rightarrow \mathbf B(\H)$, given by
\begin{equation}\label{proj}
[W(x,\xi)u](y):=e^{i(y-x/2)\cdot\xi}u(y-x),\quad\ x,y\in\X,\ \xi\in\X^*,\ u\in\H
\end{equation}
and verifying
\begin{equation}\label{project}
W(X)W(Y)=\exp(i/2[\![X,Y]\!])W(X+Y),\quad\ \forall\,X,Y\in\Xi.
\end{equation}
In terms of $P=(P_1=-i\partial_1,\cdots P_n=-i\partial_n)$ and $Q=(Q_1,\dots,Q_n)$, the usual momentum and position operators
in $\H$, one has $W(x,\xi)=e^{-\frac{i}{2}x\cdot\xi}e^{iQ\cdot\xi}e^{-ix\cdot P}$.
Associated to $W$, one has a (true) action of $\,\Xi$ by automorphisms of the $C^*$-algebra $\mathbf B(\H)$ given by
\begin{equation}\label{auto}
\mathbf T_X(S):=W(X)SW(-X),\ \quad X\in\Xi,\ S\in\mathbf B(\H).
\end{equation}
It is not norm continuous, so it defines a proper $C^*$-subalgebra
\begin{equation}\label{cvect}
\mathbf B^0(\H):=\{S\in\mathbf B(\H)\mid X\mapsto\mathbf T_X(S)\in\mathbf B(\H)\ {\rm is}\ \parallel\cdot\parallel\!
{\rm -continuous}\}.
\end{equation}

{\it The Fr\'echet filter}, denoted conveniently by $\infty$, is composed of the
complements of all the relatively compact subsets of $\,\Xi$. We recall \cite{Bo} that the filters are partially ordered
by inclusion and that an ultrafilter is a maximal filter, i.e. a filter $\mathcal F$ that is not strictly contained in another;
equivalently, for any set $A$ one should have either $A\in\mathcal F$ or $A^c\in\mathcal F$.
Let us denote by $\delta(\Xi)$ the family of all
ultrafilters on $\Xi$ that are finer than the Fr\'echet filter. Our main result is

\begin{thm}\label{main}
Let $H$ be a self-adjoint operator in $\H$ affiliated to $\mathbf B^0(\H)$. One has
\begin{equation}\label{spess}
\spe(H)=\overline{\bigcup_{\mathcal X\in\delta(\Xi)}\sp(H_\mathcal X)},
\end{equation}
where for any $\mathcal X\in\delta(\Xi)$ one sets $H_\mathcal X:=\underset{X\rightarrow\mathcal X}{\lim}\mathbf T_X(H)$
in the strong resolvent sense.
\end{thm}

Theorem \ref{main} is modelled on previous results (see \cite{Da,GI1,GI3,HM,LS} and references therein) in which,
as a rule, $H$ has to be affiliated to the smaller algebra $\mathbf E(\H)$ defined in (\ref{smaller})
and having a crossed product structure. Under this assumption, its essential spectrum can be expressed using
limits along diverging ultrafilters $\chi$ in the configuration space $\X$ applied to $\mathbf T_{(x,0)}(H)$.

To be precise we speak of {\it full-space anisotropy} when our self-adjoint operator is affiliated to
$\mathbf B^0(\H)$ without being affiliated to the smaller $\mathbf E(\H)$; to express its essential spectrum
the aforementioned limits in the configuration space are not enough and the full strength of the result (\ref{spess}) is needed.
As a simple example meant to give some intuition, let $H=h(P)+V(Q)$ in $L^2(\mathbb R)$ be the sum between the convolution
operator $h(P)$ and the "potential" $V(Q)$ (operator of multiplication by the uniformly continuous function
$V:\mathbb R\rightarrow\mathbb R$). We assume that $h:\mathbb R^*\rightarrow\mathbb R$ is continuous and that
$\lim_{\xi\rightarrow\pm\infty}h(\xi)=a_{\pm}\in\mathbb R\cup\{\pm\infty\}$.
Then $H$ (self-adjoint on a natural domain) is affiliated to $\mathbf B^0(\H)$. It is full-space anisotropic
(not affiliated to $\mathbf E(\H)$) if and only if at least one of the limits $a_{\pm}$ is finite.

Some very partial information on full phase-space anisotropy is scattered through the existing publications and
our general result (\ref{spess}) is meant to answer a conjecture of Vladimir Georgescu.
Connected results can be found in \cite{Ma}, in which however ultrafilters are not used and only bounded operators are treated.

An important ingredient for proving Theorem \ref{main} is a workable understanding of the quotient
$\mathbf B^0(\H)/\mathbf C_0(\H)$, which is relevant because the essential spectrum of an element of
$\mathbf B^0(\H)$ (or of an operator affiliated to it) coincides with the spectrum of its canonical
image in $\mathbf B^0(\H)/\mathbf C_0(\H)$. Therefore we are going to prove the following compactness criterion,
which seems new. The limits are taken in the $^*$-strong topology or, equivalently, in the strict topology defined by
the essential ideal $\mathbf C_0(\H)$.

\begin{prop}\label{dator}
An element $S$ of $\ \mathbf B^0(\H)$ is a compact operator if and only if
$\,\underset{X\rightarrow\infty}{\lim}\mathbf T_X(S)=0$ or if
and only if $\,\underset{X\rightarrow\mathcal X}{\lim}\mathbf T_X(S)=0\,$ for all $\mathcal X\in\delta(\Xi)$.
\end{prop}

The proof of Proposition \ref{dator} as well as certain examples to which (\ref{spess}) could be applied need the
Weyl pseudodifferential calculus \cite{Fol}, representing operators $S$ as quantizations $\Op(f)$ of functions
defined on phase space. Some useful facts about the Weyl calculus are reviewed in section \ref{sectra}.

The main feature that makes $\mathbf B^0(\Xi)$ treatable is the fact that {\it it is obtained by applying $\Op$
to the Rieffel deformation of the
Abelian $C^*$-algebra $\B^0(\Xi)$ of all bounded uniformly continuous functions on $\Xi$}. The Rieffel deformation
\cite{Rie1} is a general form of symbolic calculus associated to actions of vector groups (as $\Xi$) on $C^*$-algebras.
Although for $\B^0(\Xi)$ one actually gets the usual Weyl symbolic calculus, the approach in \cite{Rie1}
has many technical advantages. We review it briefly in section \ref{sectoral}.

In section \ref{secur} we use all the previous information to prove the compactness criterion, the embedding
of the quotient $\mathbf B^0(\H)/\mathbf C_0(\H)$ into a direct product $C^*$-algebra and, as a simple consequence,
Theorem \ref{main}. I am grateful to an anonymous referee for his/her advice, that lead to a simplification and
a clarification of these proofs.

Then we indicate briefly some extensions connected to Theorem \ref{main}.

The last two sections are dedicated to examples. Roughly, the new operators one expects to cover by this phase-space
anisotropic formalism are zero order pseudodifferential operators and classes of strictly positive order
non-elliptic operators.

We mention that many of the recent articles treating the essential spectrum of anisotropic operators have as a background
an Abelian locally compact group $\X$ \cite{GI1,GI3,Ma1}, or even rather general metric spaces $\X$ without a group structure
\cite{Da,Ge}. As mentioned before, the results are essentially confined to the restricted configurational isotropy due to
the use of crossed products. Rieffel's calculus has been partially extended in \cite{Ka} to actions of Abelian
locally compact groups on $C^*$-algebras and this could probably
be used with extra effort to treat operators with a complicated phase-space behavior in such a framework.

This short paper is not the right opportunity to draw the history of studying the essential spectrum with
(or without) algebraic techniques. Beside the articles already quoted, we send also to \cite{ABG,LN,MPR,RRR,RRS}
and to references therein for other results.

%-------------------------------------------------------------------------------------------------------
\section{Rieffel calculus}\label{sectoral}
%-------------------------------------------------------------------------------------------------------

Rieffel deformation \cite{Rie1} is an exact functor between categories of $C^*$-dynamical systems with group $\mathbb R^d$.
Reducing the generality to fit to the present framework, assume that $(\A,\Th,\Xi)$ is a $C^*$-dynamical system,
i.e. the vector group $\Xi$ acts strongly continuously by automorphisms on the $C^*$-algebra $\A$.
On the $C^\infty$ vectors $\Ai$ of the action one uses the
symplectic form on $\Xi$ to deform the initial product to a new one (oscillatory integrals)
\begin{equation}\label{iffel}
f\#g:=2^{2n}\int_\Xi\int_\Xi dY dZ e^{2i[\![Y,Z]\!]}\Th_Y(f)\Th_Z(g).
\end{equation}
Keeping the same involution, one gets a $^*$-algebra structure on $\Ai$ which can be completed under a $C^*$-norm
by techniques involving Hilbert modules. The action $\Th$, restricted to $\Ai$, extends to an action  of $\Xi$
on the resulting $C^*$-algebra $\A^R$ that will
be denoted by $\Theta^R$. The new space of smooth vectors $(\A^R)^\infty$ actually coincides with $\Ai$ cf.
\cite[Th. 7.1]{Rie1}, and even the natural Fr\'echet topologies on this space are the same.
We mean by this that the family of semi-norms
\begin{equation}\label{semicar}
\parallel f\parallel_\A^{(j)}:=\sum_{|\alpha|\le
j}\frac{1}{|\alpha|!}\parallel\partial_X^\alpha\left[\Th_X(f)\right]_{X=0}\parallel_\A\,,\ \ \ \ \ \ j\in\N
\end{equation}
is equivalent to the one given by an analogous expression with $\p\cdot\p_\A$ replaced by $\p\cdot\p_{\A^R}$.

The correspondence $\A\mapsto\A^R$ can be raised to a correspondence between equivariant morphisms,
cf \cite[Th. 5.7]{Rie1}: If $(\A,\Th,\Xi)$
and $(\B,\Gamma,\Xi)$ are $C^*$-dynamical systems and $\P:\A\rightarrow\B$ is a morphism satisfying
$\Gamma_X\circ\P=\P\circ\Th_X$ for any $X\in\Xi$, it restricts to a map $\P:\Ai\rightarrow\B^\infty$
which then extends to a morphism $\P^R:\A^R\rightarrow\B^R$.
We emphasize that on the common dense $^*$-subalgebra $(\A^R)^\infty=\Ai$
the actions and the morphisms coincide: $\Th^R_X|_{\Ai}=\Th_X|_{\Ai}$ and $\P^R|_{\Ai}=\P|_{\Ai}$.

Equally important \cite[Prop. 5.9]{Rie1}, any (closed two-sided) ideal $\K$ of $\A$ which is invariant under the action
$\Th$ is converted by deformation into an invariant ideal $\K^R$ of $\A^R$.

We now describe the Rieffel quantization of an intersection of ideals. For any element $j$ of a set $J$
we are given a $\Th$-invariant ideal $\K_j$ of $\A$; thus we also have the $\Th^R$-invariant ideal $\K_j^R$ of $\A^R$.

\begin{lem}\label{cannaa}
One has $\,\[\bigcap_j\K_j\]^R=\,\bigcap_j\K_j^R$\,.
\end{lem}

\begin{proof}
Both sides are $\Th^R$-invariant (closed bi-sided) ideals in $\A^R$. It will be enough to check that their $^*$-subalgebras
of smooth vectors coincide. Using the results mentioned before in this section, one can write:
$$
\big(\big[\bigcap_j\K_j\big]^R\big)^\infty=\big(\bigcap_j\K_j\big)^\infty=\bigcap_j\K_j^\infty=
\bigcap_j\big(\K_j^R\big)^\infty=\big(\bigcap_j\K_j^R\big)^\infty
$$
and we are done.
\end{proof}

\begin{rem}\label{simica}
Rieffel deformation is an almost symmetric procedure. Applying it to $\A^R$ but with the symplectic form $[\![\cdot,\cdot]\!]$
replaced by $-[\![\cdot,\cdot]\!]$, one recovers the initial $C^*$-algebra $\A$.
This follows from \cite[Th. 7.5]{Rie1}.
\end{rem}

The relevant example for us is $\A=\B^0(\Xi)$, the $C^*$-algebra of all bounded uniformly continuous functions on $\Xi$,
acted continuously by $\Xi$ by translations ($\Th=\mathcal T$):
$$
f(\cdot)\rightarrow[\mathcal T_X(f)](\cdot):=f(\cdot-X)\,\ \quad X\in\Xi.
$$
In this case $\Ai=:\B^\infty(\Xi)$ is formed of all the $C^\infty$ functions $f:\Xi\rightarrow\mathbb C$ with all
the partial derivatives bounded; the traditional notation in pseudodifferential theory is $S^0_{0,0}(\Xi)$.
On $\B^\infty(\Xi)$ Rieffel's composition law $\#$ coincides with the Weyl multiplication $\sharp$; see \cite{Fol}.

Rieffel's deformation of $\B^0(\Xi)$ will be denoted by $\mathfrak B^0(\Xi)$;
it forms an operator algebra extension of the zero order pseudodifferential symbols, having full phase-space anisotropy.
Elements of the H\"ormander spaces $S^{-m}_{\rho,\delta}(\Xi)$, $m>0$ of strictly negative order
could be considered trivial at infinity with respect to
$\xi\in\X^*$, having interesting (anisotropic) asymptotic behavior only in $x\in\X$; they generate the
$C^*$-algebra $\mathfrak E(\Xi)$ of Remark \ref{doica}.

We are going to denote by $\mathfrak T:=\mathcal T^R$ the action of $\Xi$ on $\mathfrak B^0(\Xi)$ obtained from $\mathcal T$ by
Rieffel deformation. But it is easy to see that $\mathfrak B^0(\Xi)$ is entirely composed of temperate distributions and that
$\mathfrak T_X$ is just translation with $X$ restricted from the dual of the Schwartz space (see below).

%-------------------------------------------------------------------------------------------------------
\section{Hilbert space representations}\label{sectra}
%-------------------------------------------------------------------------------------------------------

We recall some basic facts about the Weyl calculus. A correspondence between functions (and distributions)
$f$ on the phase space $\Xi$ and operators $\Op(f)$ acting on functions on the configuration space $\X$ is given formally by
\begin{equation}\label{op}
[\Op(f)u](x):=\int_\X dy\int_{\X^*}\!\! d\eta\,e^{iy\cdot\eta}f\left(\frac{x+y}{2},\eta\right)u(y).
\end{equation}
Various interpretations \cite{Fol} can be given to (\ref{op}) under various assumptions on $f$ and $u$.
We notice only that $\Op$ defines an isomorphism between the space of tempered distributions $\S'(\Xi)$ and the
space $\mathbf L[\S(\X);\S'(\X)]$ of linear continuous operators from the Schwartz space $\S(\X)$ to its dual $\S'(\X)$.
It also restricts to an isomorphism $\Op:\S(\Xi)\rightarrow\mathbf L[\S'(\X);\S(\X)]$. On various subspaces of $\S'(\Xi)$
one introduces the multiplication $\sharp$ (Weyl composition) satisfying $\Op(f)\Op(g)=\Op(f\sharp g)$.
One of these spaces is $\S(\Xi)$, a (Fr\'echet) $^*$-algebra under $\sharp$ and complex conjugation.

It is easy to show that any $\mathbf T_X$ (introduced at (\ref{auto})) will define automorphisms of
$\,\mathbf L[\S(\X),\S'(\X)]$ and of $\,\mathbf L[\S'(\X),\S(\X)]$. The next relation, easy to check on $\S'(\Xi)$, is basic:
\begin{equation}\label{basic}
\mathbf T_X\circ\Op=\Op\circ\mathcal T_X,\ \quad X\in\Xi.
\end{equation}
When written on the subspace $\mathfrak B^0(\Xi)$, the automorphism $\mathcal T_X$ can be replaced by $\mathfrak T_X$.

Since $\B^0(\Xi)$ possesses the essential invariant ideal $\C_0(\Xi)$ of continuous functions on $\Xi$ that are small
at infinity, one gets by deformation \cite[Prop. 5.9]{Rie1} an essential invariant ideal $\mathfrak C_0(\Xi)$
inside $\mathfrak B^0(\Xi)$. On $\mathfrak B^0(\Xi)$ the seminorms
\begin{equation}\label{strict2}
\left\{\p f\p_{\mathfrak B^0(\Xi)}^h:=\,\p f\sharp \,h\p_{\mathfrak B^0(\Xi)}+
\p h\sharp \,f\p_{\mathfrak B^0(\Xi)}\ \mid\ h\in\mathfrak C_0(\Xi)\right\}
\end{equation}
define the strict topology associated to the essential ideal $\mathfrak C_0(\Xi)$. We are going to denote by
$\mathfrak B^0(\Xi)_{{\rm str}}$ the space $\mathfrak B^0(\Xi)$ endowed with this topology.
Let us also set $\mathbf B^0(\H)_{{\rm str}}$ for the space $\mathbf B^0(\H)$ with the strict topology associated
to the essential ideal $\mathbf C_0(\H)$ of compact operators on $\H$, via the family of seminorms
\begin{equation}\label{strict3}
\left\{\p S\p_{\mathbf B(\H)}^K:=\,\p KS\p_{\mathbf B(\H)}+\p SK\p_{\mathbf B(\H)}\,\mid K\in\mathbf C_0(\H)\right\}.
\end{equation}

\begin{prop}\label{precisely}
\begin{enumerate}
\item
$\Op$ realizes a $C^*$-isomorphism between $\mathfrak{B\,}^0(\Xi)$ and $\,\mathbf B^0(\H)$.
\item
The image of $\,\mathfrak C_0(\Xi)$ through $\Op$ is precisely $\mathbf C_0(\H)$.
\item
The mapping $\Op:\mathfrak B^0(\Xi)_{{\rm str}}\rightarrow \mathbf B^0(\H)_{{\rm str}}$ is an isomorphism.
\end{enumerate}
\end{prop}

\begin{proof}
The $C^*$-algebra $\mathfrak{B\,}^0(\Xi)$ contains the $^*$-subalgebra $\B^\infty(\Xi)$ densely.
By the Calderon-Vaillancourt Theorem \cite{Fol}, $\Op:\B^\infty(\Xi)\rightarrow\mathbf B(\H)$ is a well-defined representation.
In \cite[Prop. 2.6]{Ma} it is shown that it extends to a faithful representation
$\Op:\mathfrak{B\,}^0(\Xi)\rightarrow\mathbf B(\H)$.
(The isometry of $\Op$ with respect to the Rieffel norm $\p\cdot\p_{\mathfrak{B\,}^0(\Xi)}$ is also proven in a different way
in \cite{BM}.) Then the relation (\ref{basic}) and the surjectivity of $\Op:\S'(\Xi)\rightarrow\mathbf L[\S(\X);\S'(\X)]$
easily leads to $\Op\[\mathfrak B^0(\Xi)\]=\mathbf B^0(\H)$.

The second point follows from the fact that $\Op[\S(\Xi)]$ is dense in $\mathbf C_0(\H)$;
use also the density of $\S(\Xi)$ in the Fr\'echet topology of $\C_0(\Xi)^\infty=\mathfrak C_0(\Xi)^\infty$,
which is dense in $\mathfrak C_0(\H)$.

The third statement should already be clear. Working with the seminorms for instance, one shows immediately that
$\parallel \!\Op(f)\!\parallel_{\mathbf B^0(\H)}^{\Op(h)}=\,\p \!f\!\p_{\mathfrak B^0(\Xi)}^h$
for $f\in \mathfrak B^0(\Xi)$ and $h\in \mathfrak C_0(\Xi)$.
This follows from the definitions, from the points 1 and 2 and from the relations
$\Op(f)\Op(h)=\Op(f\sharp \,h)$ and $\Op(h)\Op(f)=\Op(h\sharp \,f)$.
\end{proof}

%-------------------------------------------------------------------------------------------------------
\section{Proofs}\label{secur}
%-------------------------------------------------------------------------------------------------------

An ingredient for proving Theorem \ref{main} is

\begin{prop}\label{simon}
Let $S\in\mathbf B^0(\H)$ and let $\mathcal U$ be an ultrafilter on $\Xi$. Then
$\mathbf T_\mathcal U(S):=\underset{X\rightarrow\mathcal U}{\lim}\mathbf T_X(S)$
exists in the $\mathbf C_0(\H)$-strict topology or, equivalently, in the $^*$-strong topology.

It defines a morphism $\mathbf T_\mathcal U:\mathbf B^0(\H)\rightarrow\mathbf B^0(\H)$.
\end{prop}

Before starting the proof we must recall a criterion of compactness due to Riesz and Kolmogorov, in the form
\cite[Th.3.4]{GI2} needed here: {\it A bounded subset $M$ of $\H=L^2(\X)$ is relatively compact if and only if
$\underset{Y\rightarrow 0}{\lim}\,\underset{v\in M}{\sup}\p [W(Y)-1]v\p=0$.}

\begin{proof}
By Lemma C.6 in \cite{RaW}, on norm-bounded subsets of $\mathbf B(\H)$ the $\mathbf C_0(\Xi)$-strict topology
coincides with the $^*$-strong topology, which will be used below.

From (\ref{auto}) and (\ref{project}) it follows that
\begin{equation}
W(Y)\mathbf T_X(S)=\mathbf T_{X+Y}(S)W(Y),\quad\ \forall\,X,Y\in\Xi,
\end{equation}
which implies that
$$
W(Y)\mathbf T_X(S)-\mathbf T_X(S)=[\mathbf T_{X+Y}(S)-\mathbf T_X(S)]W(Y)+\mathbf T_X(S)[W(Y)-1].
$$
Pick a vector $u\in\H$, recall that $W(\cdot)$ is strongly continuous, $S$ belongs to $\mathbf B^0(\H)$ and
$\mathbf T_{X+Y}(S)=\mathbf T_X[\mathbf T_Y(S)]$. Then immediately
\begin{equation}\label{imedi}
\lim_{Y\rightarrow 0}\sup_{X\in\Xi}\p [W(Y)-1]\mathbf T_X(S)u\p\,=0,
\end{equation}
implying by the Riesz-Kolmogorov criterion that the bounded set $M:=\{\mathbf T_X(S)u\mid X\in\Xi\}$ is relatively compact in $\H=L^2(\X)$.
This can be done also for $S^*\in\mathbf B^0(\H)$. It follows that
$\mathbf T_\mathcal U(S):=\mathbf C_0\!-\!\!\underset{X\rightarrow\mathcal U}{\lim}\mathbf T_X(S)\in\mathbf B^0(\H)$ exists $^*$-strongly
for every ultrafilter $\mathcal U$, in particular for the elements of $\delta(\Xi)$.

It is easy to see that it defines a morphism $\mathbf T_\mathcal U:\mathbf B^0(\H)\rightarrow\mathbf B^0(\H)$.
\end{proof}

We continue by proving Proposition \ref{dator}, relying partly on the techniques developed in \cite{Rie1} suitably adapted
to our setting and notations.

\begin{proof}
We already know from Proposition \ref{simon} that for $S\in\mathbf B^0(\H)$ and $\mathcal X\in \delta(\Xi)$
the limit $\mathbf T_\mathcal X(S)$ exists. Since every filter is the intersection of the ultrafilters containing it,
then $S\in\bigcap_{\mathcal X\in\delta(\Xi)}\ker[\mathbf T_\mathcal X]$ if and only if
$\mathbf C_0\!-\!\!\underset{X\rightarrow\infty}{\lim}\mathbf T_X(S)=0$.

By taking into account (\ref{basic}) and Proposition \ref{precisely}, it remains to show for an element
$f\in\mathfrak B^0(\Xi)$ that $f\in\mathfrak C_0(\Xi)$ if and only if
$\mathfrak T_\mathcal X(f):=\,\mathfrak C_0\!-\!\underset{X\rightarrow\mathcal X}{\lim}\mathfrak T_X(f)=0$
for all $\mathcal X\in\delta(\Xi)$.

We learn from \cite[Sect. 5.1]{GI3} that an element $f\in\B^0(\Xi)$ belongs to $\C_0(\Xi)$ iff
$\,\C_0\!-\!\underset{X\rightarrow\mathcal X}{\lim}\mathcal T_X(f)=0$ for all $\mathcal X\in\delta(\Xi)$\,.
We referred to the limit in the $\C_0(\Xi)$-strict topology of $\B^0(\Xi)$, defined by the semi-norms
\begin{equation}\label{endless}
\left\{\p f\p_{\B^0(\Xi)}^h:=\,\p hf\p_{\B^0(\Xi)}\,\mid h\in\C_0(\Xi)\right\}.
\end{equation}
In (\ref{endless}) one could use only smooth and compactly supported elements $h\in C^\infty_{{\rm c}}(\Xi)$
and one gets actually convergence which is uniform on compact subsets of $\Xi$.
Taking also Lemma \ref{cannaa} into account, it is enough to show that $\ker[\mathfrak T_\mathcal X]$
is the Rieffel deformation of $\ker[\mathcal T_\mathcal X]$, which would follow from
$\ker[\mathfrak T_\mathcal X]^\infty=\ker[\mathcal T_\mathcal X]^\infty$
(and actually this later equality would be enough to finish the proof).

Let us fix $f\in\ker[\mathcal T_\mathcal X]^\infty$, which membership is equivalent to
$\,\partial^\gamma f\in\ker[\mathcal T_\mathcal X]$ for all $\gamma\in\N^{2n}$. This means
$$
\lim_\mathcal X\p h\,\mathcal T_X(\partial^\gamma f)\p_{\mathcal B^0(\Xi)}\,=\,
\lim_{\mathcal X}\p \mathcal T_{-X}(h)\,\partial^\gamma f\p_{\mathcal B^0(\Xi)}=0
$$
for all $\gamma\in\N^{2n}$ and $h\in\C_{{\rm c}}^\infty(\Xi)$. Now consider $\alpha,\gamma\in\N^{2n}$
and $h\in\C_{{\rm c}}^\infty(\Xi)$ fixed; one has
$$
\p \partial^\alpha\left[h\,\mathcal T_X(\partial^\gamma f)\right]\p_{\mathcal B^0(\Xi)}\,\le
\sum_{\beta\le\alpha}C_{\alpha,\beta}\p \,\mathcal T_{-X}\(\partial^{\alpha-\beta}h\)\partial^{\beta+\gamma}
f\p_{\mathcal B^0(\Xi)}\,\underset{X\rightarrow\mathcal X}{\longrightarrow}0\,.
$$
This means that $\mathcal T_{-X}(h)\,\partial^\gamma f$ converges to $0$ in the Fr\'echet topology of $\B^\infty(\Xi)$.
From \cite[Prop. 4.13]{Rie1}
it will follow that $\p\mathcal T_{-X}(h)\,\sharp\,\partial^\gamma f\p_{\mathfrak B^0(\Xi)}\,=
\,\p h\,\sharp\,\mathfrak T_{X}(\partial^\gamma f)\p_{\mathfrak B^0(\Xi)}$
converges to zero when $X\rightarrow\mathcal X$. We get $\partial^\gamma f\in\ker[\mathfrak T_\mathcal X]$ for every
$\gamma\in\mathbb N^{2n}$, meaning that $f\in\ker[\mathfrak T_\mathcal X]^\infty$.

For the opposite inclusion $\ker[\mathfrak T_\mathcal X]^\infty\subset\ker[\mathcal T_\mathcal X]^\infty$
one uses Remark \ref{simica}.
\end{proof}

\begin{rem}\label{discrepanta}
Actually \cite[Prop. 4.13]{Rie1} refers to nets. One can rephrase it for filters, by suitable modifications.
On the other hand, there is a simple way to pass from filters to nets and conversely, preserving convergence. In fact this
is also a useful device if one wants to rewrite Theorem \ref{main} in terms of diverging nets on $\Xi$.
\end{rem}

\begin{rem}\label{repanta}
It is useful and interesting to record the present form of the proof of Proposition \ref{simon}, due to Vladimir Georgescu,
which does not depend on the pseudodifferential calculus. But with some more work, one could show that
$\mathfrak T_\mathcal X$ is the Rieffel deformation of the morphism $\mathcal T_\mathcal X$ for any ultrafilter $\mathcal X$.
Then one could just push the morphisms $\mathcal T_\mathcal X$
(known to exist and useful anyhow to characterize the ideal $\C_0(\Xi)\subset\B^0(\Xi)$) through the
formalism, getting successively $\mathfrak T_\mathcal X$ and $\mathbf T_\mathcal X$. A better option would be to
preserve Proposition \ref{simon} as it is and to obtain Proposition \ref{dator} in some direct way.
\end{rem}

\begin{cor}\label{mortar}
The quotient $\mathbf B^0(\H)/\mathbf C_0(\H)$ embeds canonically  as a $C^*$-subalgebra of
$\,\prod_{\mathcal X\in\delta(\Xi)}\mathbf B^0(\H)$,
where the sign $\prod$ denotes a restricted product: its elements are families with a uniform bound on the norms.
\end{cor}

\begin{proof}
The kernel of the product morphism
$
\(\mathbf T_\mathcal X\)_{\mathcal X\in\delta(\Xi)}:\mathbf B^0(\H)\rightarrow\prod_{\mathcal X\in\delta(\Xi)}\mathbf B^0(\H)
$
coincides with $\bigcap_{\mathcal X\in\delta(\Xi)}\ker[\mathbf T_\mathcal X]$, which equals $\mathbf C_0(\Xi)$
by Proposition \ref{dator}. Then from a simple abstract argument it follows that
$\mathbf B^0(\H)/\mathbf C_0(\H)\hookrightarrow\prod_{\mathcal X\in\delta(\Xi)}\mathbf B^0(\H)\,$.
\end{proof}

Now Theorem \ref{main} follows easily. The essential spectrum of $H$ coincides with the spectrum of its image
(expressed at the level of resolvents) in the quotient $\mathbf B^0(\H)/\mathbf C_0(\H)$. This one can be computed in the
product $\,\prod_{\mathcal X\in\delta(\Xi)}\mathbf B^0(\H)$, so it is the closed union of spectra of all the components.
Some of the self-adjoint operators $H_\mathcal X$ might not be densely defined.

%---------------------------------------------------------------------------------
\section{Some comments and extensions}\label{joioasa}
%-------------------------------------------------------------------------------------

\begin{rem}\label{unica}
There is a certain redundancy in (\ref{spess}). Two ultrafilters $\mathcal X$ and $\mathcal X'$ would give the same operator
$H_\mathcal X=H_{\mathcal X'}$ if they have the same envelope. The envelope $\mathcal X^\circ$ of $\mathcal X$ is
the filter generated by sets $A+V$ where $A\in\mathcal X$ and $V$ is a neighborhood of $0\in\Xi$.
This is explained in \cite[2.6]{GI3} in a different but connected setting.
\end{rem}

\begin{rem}\label{zerica}
One can use (\ref{spess}) to study the essential spectrum of self-adjoint operators affiliated to unital
$C^*$-subalgebras $\mathbf A$ of $\mathbf B^0(\H)$ which are invariant under the automorphisms $\mathbf T_X$,
by the same techniques as in sections 2.5 and 5.3 from \cite{GI3}; see also \cite{Ma}. Such algebras would induce a
rougher equivalence relation on the set $\delta(\Xi)$ then the one hinted in
Remark \ref{unica}. More precise information about the limits $H_{\mathcal X}$ would also be available.
So one could adapt to phase space
concrete types of anisotropy as those investigated in configuration space in references as \cite{Da,GI1,GI3,LS,Ma1}.
\end{rem}

\begin{rem}\label{doica}
The most efficient $C^*$-algebras considered until now in connection with the investigation of the essential spectrum
of anisotropic operators on $\mathbb R^n$ are $C^*$-subalgebras of
\begin{equation}\label{smaller}
\mathbf E(\H):=\left\{S\in\mathbf B^0(\H)\mid\ \parallel W(x,0)S^{(*)}-S^{(*)}\parallel_{\mathbf B(\H)}
\underset{x\rightarrow 0}{\longrightarrow}0\right\}.
\end{equation}
(The notation means that the condition is fulfilled both for $S$ and $S^*$.)
It is clear that $\mathbf E(\H)$ is an ideal in $\mathbf B^0(\H)$. It is known \cite{GI1,GI3} that
$\mathfrak E(\Xi):=\Op^{-1}[\mathbf E(\H)]$ coincides with the crossed product $\B^0(\X)\rtimes\X$ and it is also easy
to see that it is the Rieffel deformation of $\B^0(\X)\otimes \C_0(\X^*)$.
They played a privileged role in \cite{GI1,GI3,Ma1} (even for Abelian locally compact groups $\X$) in the study of
the essential spectrum of $\X$-anisotropic operators in $\H=L^2(\X)$, but they are not enough to cover phase-space anisotropy.
\end{rem}

\begin{rem}\label{treica}
Another natural ideal of $\mathbf B^0(\H)$ is
$$
\mathbf F(\H):=\{S\in\mathbf B^0(\H)\mid\ \parallel W(0,\xi)S^{(*)}-S^{(*)}\parallel_{\mathbf B(\H)}
\underset{\xi\rightarrow 0}{\longrightarrow}0\},
$$
for which obvious assertions can be made by analogy with $\mathbf E(\H)$, both concerning the structure and the usefulness.
The essential spectrum of self-adjoint operators $H$ affiliated to $\mathbf F(\H)$ would involve strong resolvent limits of
$\mathbf T_{(0,\xi)}(H)$ along ultrafilters finer than the Fr\'echet filter in the momentum space $\X^*$.

As a consequence of the Riesz-Kolmogorov criterion, one has $\mathbf E(\H)\cap\mathbf F(\H)=\mathbf C_0(\H)$.
\end{rem}

%---------------------------------------------------------------------------------
\section{Affiliation}\label{joasa}
%-------------------------------------------------------------------------------------

We give explicit affiliation criteria to the $C^*$-algebras $\mathfrak B^0(\Xi)$ and $\mathbf B^0(\H)$.
Some of them are (almost) obvious, others are rather simple adaptations of results from previous articles (mostly \cite{GI3}),
so we present them as a sequence of examples. It goes without saying that all the operators proven previously
(as in \cite[Sect.4]{GI3}) to be affiliated to $\mathbf E(\H)$ are also affiliated to $\mathbf B^0(\H)$.

{\bf A.} Clearly, every self-adjoint element of $\mathbf B^0(\H)$ is affiliated to $\mathbf B^0(\H)$. This includes,
for instance, operators of the form $\Op(f)$, with $f\in \B^\infty(\Xi)_\mathbb R$.
Other examples are $\varphi(Q)$ or $\psi(P)$ with $\varphi\in \B^0(\X)_\mathbb R$ and $\psi\in \B^0(\X^*)_\mathbb R$
or self-adjoint linear combinations of products of such operators.

{\bf B.} If $H_0$ is already shown to be affiliated, obviously $H=H_0+H_1$ will be affiliated too for any
$H_1\in\mathbf B^0(\H)$.
Assume for instance that $\Op(f_0)$ is affiliated to $\mathbf B^0(\H)$. The same will be true for $\Op(f_0+f_1)$ for any real
$f_1\in \B^\infty(\Xi)$. In particular this happens for $H_1=\lambda\in\mathbb R$, so the affiliation to
$\mathbf B^0(\H)$ of lower bounded operators $H$ can be reduced to the case $H\ge 1$.

{\bf C.} For a real function $a$ defined on $\X^*$, the convolution operator $a(P)$ is affiliated to $\mathbf B^0(\H)$
if and only if the function $(a+i)^{-1}$ is uniformly continuous, since
$\mathbf T_{(x,\xi)}\left[(a(P)+i)^{-1}\right]=(a(P+\xi)+i)^{-1}$. Thus one needs to check that
$$
\sup_{\eta\in\X^*}\frac{|a(\eta+\xi)-a(\eta)|}{(1+|a(\eta+\xi)|)(1+|a(\eta)|)}\underset{\xi\rightarrow 0}{\longrightarrow}0\,.
$$
This happens, of course, when $a\in \B^0(\X^*)$, or when $a$ is proper (diverges at infinity),
since in this second case $(a+i)^{-1}\in \mathbb C\otimes C_0(\X^*)$ and $a(P)$ will even be affiliated
to $\mathbf E(\H)$. There are, of course, many other opportunities for $(a+i)^{-1}$ to be uniformly continuous.
Assume for instance, as in \cite[4.2]{GI3}, that $a$ is $C^1$ and equivalent to a weight.
If one has $|a'|\le C(1+|a|)$ for some constant $C$, then
$(a+i)^{-1}$ is indeed uniformly continuous. For criteria involving higher order derivatives, see \cite[Ex.\,4.17]{GI3}.
Let us use a decomposition $\X^*=\X^*_1\times\dots\times\X^*_m$ and pick real numbers $s_1,\dots,s_m$. The function
$a(\xi):=\langle\xi_1\rangle^{s_1}\dots\langle\xi_m\rangle^{s_m}$ leads to an operator $a(P)$ affiliated to $\mathbf B^0(\H)$
independently of the signs of $s_1,\dots,s_m$. Another interesting example is
$a(\xi):=\exp(s_1\xi_1+\dots+s_n\xi_n)$ in $\X^*=\mathbb R^n$.
Many other very anisotropic combinations are possible, going far beyond ellipticity.

{\bf D.} Similar statements hold for the multiplication operator $b(Q)$. Of course this follows directly, since
$$\mathbf T_{(x,\xi)}\left[(b(Q)+i)^{-1}\right]=(b(Q+x)+i)^{-1},$$ but can also be deduced from a general symmetry
principle: Assume that $f$ is affiliated to $\mathfrak B^0(\Xi)$ and identify $\X^*$ with $\X$. Then the function
$f^\circ(x,\xi):=f(\xi,x)$ is also affiliated to $\mathfrak B^0(\Xi)$.

{\bf E.} Let $H$ be a self-adjoint operator in $\H$ with domain $\mathcal E$ endowed with the graph norm.
Denoting by $\mathcal E^*$ the (anti-)dual of $\mathcal E$, one gets canonical embeddings
$\mathcal E\hookrightarrow\H\hookrightarrow\mathcal E^*$.
Assume that $W(X)\mathcal E\subset\mathcal E,\ \forall X\in\Xi$. Then $H$ is affiliated to $\mathbf B^0(\H)$ if and only if
$\,\p [W(X),H]\p_{\mathbf B(\mathcal E,\mathcal E^*)}\underset{X\rightarrow 0}{\longrightarrow} 0$.

{\bf F.} If only the form domain $\mathcal G$ of the self-adjoint operator $H$ is invariant under $W$, then the relation
$$\,\p [W(X),H]\p_{\mathbf B(\mathcal G,\mathcal G^*)}\equiv\,
\p \mathbf T_X(H)-H\p_{\mathbf B(\mathcal G,\mathcal
G^*)}\underset{X\rightarrow 0}{\longrightarrow} 0$$ would imply
that $H$ is affiliated to the $C^*$-algebra $\mathbf B^0(\H)$.

See \cite[Def.\,4.7,\,\,Cor.\,4.8,\,\,Prop.\,4.9]{GI3} for the affiliation of abstract operators defined as
form-sums $H=H_0+H_1$.

%-------------------------------------------------------------------------------------------------------
\section{Second order differential operators}\label{seretide}
%-------------------------------------------------------------------------------------------------------

We are interested in partial differential operators in $\H=L^2(\mathbb R^n)$ which are
defined formally as $H_a:=\sum_{j,k=1}^n P_ja_{jk}(Q)P_k$.
Perturbations (especially by multiplication operators) can be added by the results reviewed in Section \ref{joasa}.
It will always be assumed that the matrix $(a_{jk}(x))$ is positive definite and given by $L^1_{{\rm loc}}$-functions.
Defining the quadratic form $q^{(0)}_a$ on $\C_c^\infty(\X)$ (the smooth compactly supported functions on $\X=\mathbb R^n$) by
\begin{equation*}\label{quadratic}
q^{(0)}_a(u):=\int_{\mathbb R^n}dx\sum_{j,k=1}^n a_{jk}(x)\overline{(\partial_j u)(x)}(\partial_k u)(x),
\end{equation*}
we are also going to suppose that this quadratic form is closable. Generous explicit conditions on $a$ insuring
this can be found in \cite{Dav,RW}.

We define a norm on $\C_c^\infty(\X)$ by
$\p u\p_a:=\left(q^{(0)}_a(u)+\p u\p^2\right)^{1/2}$ and denote by $\mathcal G_a$ the Hilbert space obtained by completing
$\C_c^\infty(\X)$ with respect to $\p \cdot\p_a$\,. One has canonically
$\mathcal G_a\hookrightarrow\H\hookrightarrow\mathcal G^*_a$
and $q_a^{(0)}$ extends to a closed form $q_a:\mathcal G_a\rightarrow[0,\infty)$. A unique self-adjoint
positive operator $H_a$ is assigned to $q_a$, with $D(H_a^{1/2})=\mathcal G_a$ and
$\p H_a^{1/2}u\p=q_a(u)^{1/2},\,\forall u\in\mathcal G_a$;
it extends to a symmetric element of $\mathbf B(\mathcal G_a;\mathcal G_a^*)$. Just under the conditions above we say that
$H_a$ is {\it weakly elliptic}. If it is {\it uniformly elliptic} (i.e. $0<c\,{\rm id}\le a(\cdot)\le c'\,{\rm id}<\infty$),
it is known \cite{DG,GI3}  to be affiliated to $\mathbf E(\H)\subset\mathbf B^0(\H)$.

\begin{prop}\label{domeniul}
Assume that $0<a(\cdot)\le c'\,{\rm id}<\infty$ and that there is a continuous function
$C:\X\rightarrow(0,\infty)$ satisfying $C(0)=1$ such that
\begin{equation}\label{adios}
a(z+x)\le C(x)a(z),\ \quad \forall\,x,z\in\X.
\end{equation}
Then $W(X)\mathcal G_a\subset\mathcal G_a$ for all $X\in\Xi$ and $H_a$ is affiliated to $\mathbf B^0(\H)$.
\end{prop}

\begin{proof}
The first assertion is very simple to check.

\noindent
Then notice that, computing on $\C_c^\infty(\X)$, one has the identity
$$
\begin{aligned}
\mathbf T_X(H_a)-H_a=&\sum_{j,k=1}^n P_j[a_{jk}(Q+x)-a_{jk}(Q)]P_k\\
+&\sum_{j,k=1}^n\left\{\xi_ja_{jk}(Q+x)P_k+P_ja_{jk}(Q+x)\xi_k+a_{jk}(Q+x)\xi_j\xi_k\right\}.
\end{aligned}
$$
Using (\ref{adios}) it follows easily that
$$
\langle u,[\mathbf T_X(H_a)-H_a]u\rangle\le D(X)\p u\p_{\mathcal G_a}^2,\ \quad\forall\,u\in \C_c^\infty(\X)
$$
with $D(X)\rightarrow 0$ when $X\rightarrow 0$, implying that
$\p \mathbf T_X(H_a)-H_a\p_{\mathbf B(\mathcal G_a;\mathcal G_a^*)}\rightarrow 0$ when $X\rightarrow 0$.
Thus $H_a$ is affiliated to $\mathbf B^0(\H)$, by the criterion {\bf F} of the preceding Section.
\end{proof}

\begin{rem}\label{ptimal}
{\rm This is far from optimal. If the coefficients $a(x)$ grow faster than $|x|^2$ at infinity,
then $H_a$ has a compact resolvent by \cite[Cor.\,1.6.7]{Dav}, so it is affiliated to
$\mathbf C_0(\H)\subset\mathbf E(\H)\subset\mathbf B^0(\H)$.}
\end{rem}

\begin{rem}\label{nogo}
{\rm By \cite[Th. 9]{DG}, if there is a diverging sequence of points $(x_m)_{m\in\mathbb N}$ in the configuration space $\X$
and a diverging sequence $(r_m)_{m\in\mathbb N}$ of positive numbers such that
$$
\lim_{m\rightarrow\infty}\left\{\sup_{|x-x_m|\le r_m}\p a(x)\p\right\}=0,
$$
then the operator $H_a$ is not affiliated to the crossed product $C^*$-algebra $\mathbf E(\H)$.
This happens for instance if $\p a(x)\p\rightarrow 0$ when $x\rightarrow\infty$.
In a huge number of such situations (\ref{adios}) is fulfilled and one really needs ultrafilters
in phase space to describe the essential spectrum.
}
\end{rem}

%-------------------------------------------------------------------------------------------------------


\begin{thebibliography}{1}
%-------------------------------------------------------------------------------------------------------

\bibitem{ABG}{\it W. O. Amrein, A. Boutet de Monvel, V. Georgescu}.
$C_0$-Groups, Commutator Methods and Spectral Theory of N-Body Hamiltonians.\emph{ Birkh\"auser, Basel}, 1996.

\bibitem{BM} {\it I. Belti\c t\u, M. M\u antoiu}. Rieffel quantization and twisted crossed products. submitted.

\bibitem{Bo} {\it N. Bourbaki}. El\'ements de math\'ematique: topologie g\'en\'erale, Ch. 1–4, \emph{Hermann, Paris,} 1971.

\bibitem{Dav} {\it E. B. Davies}. Heat Kernels and Spectral Theory. \emph{Cambridge Univ. Press}, 1989.

\bibitem{Da} {\it E. B. Davies}. Decomposing the essential spectrum.  \emph{J. Funct. Anal.} \textbf{257} no.2,
(2009), 506--536.

\bibitem{DG} {\it E. B. Davies, V. Georgescu}. $C^*$-algebras associated with some second order differential operators.
Preprint ArXiV and to appear in J. Operator Theory.

\bibitem{Fol} {\it G. B. Folland}. Harmonic Analysis in Phase Space.
Annals of Mathematics Studies, {\textbf 122} \emph{ Princeton
University Press, Princeton, NJ}, 1989.

\bibitem{Ge} {\it V. Georgescu}. On the structure of the essential spectrum of elliptic operators in metric spaces.
\emph{J. Funct. Anal.} {\textbf 260} (2011), 1734--1765.

\bibitem{GI1} {\it V. Georgescu, A. Iftimovici}. Crossed
Products of $C^*$-algebras and spectral analysis of quantum
Hamiltonians. \emph{Commun. Math. Phys.} {\textbf 228} (2002),
519--560.

\bibitem{GI2} {\it V. Georgescu, A. Iftimovici}. Riesz Kolmogorov compactness criterion,
Lorentz convergence and Ruelle theorem on locally compact Abelian
groups. \emph{Potential Analysis} {\textbf 20}, (2004) 265–-284.

\bibitem{GI3} {\it V. Georgescu, A. Iftimovici}. Localizations at infinity and essential spectrum of quantum
Hamiltonians. I. General Theory. \emph{Rev. Math. Phys.} {\textbf
18} (4), (2006),  417--483.

\bibitem{HM} {\it B. Helffer, A. Mohamed}. Caract\'erisation du spectre essentiel de l'op\'erateur de Schr\"odinger
avec un champ magn\'etique. \emph{Ann. Inst. Fourier} {\textbf 38}
no.2 (1988), 95--112.

\bibitem{Ka} {\it P. Kasprzak}. Rieffel deformation via crossed products. \emph{J. Funct. Anal.} \textbf{257} no.5
(2009), 1228--1332.

\bibitem{LS} {\it Y. Last, B. Simon}. The essential spectrum of Schr\"{o}dinger, Jacobi and CMV operators.
\emph{J. d'Analyse Math.} {\textbf 98}  (2006), 183--220.

\bibitem{LN} {\it R. Lauter, V. Nistor}. Analysis of geometric operators on open manifolds: a Groupoid
Approach. in \emph{Quantization of Singular Symplectic Quotients},
Progr. Math., {\textbf 198} Birkh\"auser, Basel (2001), 181-229.

\bibitem{Ma1}{\it M. M\u antoiu}. Compactifications, dynamical systems at infinity and the essential spectrum of
generalized Sch\"odinger operators. \emph{J. reine angew. Math.}
{\textbf 550} (2002), 211--229.

\bibitem{Ma} {\it M. M\u antoiu}. Rieffel's pseudodifferential calculus and spectral analysis for quantum Hamiltonians.
Preprint ArXiV and to appear in \emph{Ann. Inst. Fourier}.

\bibitem{MPR} {\it M. M\u antoiu, R. Purice, S. Richard}. Spectral and propagation results
for magnetic Schr\"odinger operators; a $C^*$-algebraic framework.
\emph{J. Funct. Anal.} {\textbf 250}  (2007), 42--67.

\bibitem{Rie1} {\it M. A. Rieffel}. Deformation quantization for actions of $\,\mathbb R^d$. \emph{Memoirs AMS.}
{\textbf 506}  (1993).

\bibitem{RRR} {\it V. S. Rabinovich, S. Roch, J.Roe}. Fredholm indices of band-dominated operators.
\emph{Int. Eq. Op. Th.} {\textbf 49}  (2004), 221--238.

\bibitem{RRS}{\it V. S. Rabinovich, S. Roch, B. Silbermann}. Limit Operators and their Applications in Operator Theory.
\emph{Operator Theory: Adv. and Applic}. {\textbf 150}
Birkh\"auser, Basel, 2004.

\bibitem{RaW} {\it I. Raeburn, D. Williams}.  Morita Equivalence and Continuous-Trace $C^*$-Algebras.
\emph{Mathematical Surveys and Monographs}. {\textbf 60} American
Mathematical Society (1998).

\bibitem{RW} {\it M. R\"ockner, N. Wielens}. Dirichlet forms - closability and change of speed measure, in
\emph{Infinite Dimensional Analysis and Stochastic Processes},
Research Notes in Mathematics {\textbf 124} Pitman (1985)
119--144.

\bibitem{Wo} {\it S. L. Woronowicz}. Unbounded elements affiliated with $C^*$-algebras and non-compact quantum groups.
\emph{Commun. Math. Phys.} {\textbf 136} (1991), 399–-432.

\end{thebibliography}
\end{document}